\documentclass[12pt]{article}
\usepackage{amsfonts}
\usepackage{amsmath,amscd, amsthm, tabularx}
\newtheorem{prop}{Proposition}[section]
\newtheorem{thm}[prop]{Theorem}

\newtheorem{lemma}[prop]{Lemma}
\theoremstyle{remark}
\newtheorem{rmk}[prop]{Remark}
\theoremstyle{definition}
\newtheorem{definition}[prop]{Definition}
\usepackage{graphicx,color}
\usepackage{cite}

\title{Mirror Transitions and the Batyrev-Borisov Construction}
\author{Karl Fredrickson}
\date{}
\begin{document}
\maketitle
\begin{abstract} We consider examples of extremal transitions between families of Calabi-Yau complete intersection threefolds in toric varieties, which are induced by toric embeddings of one toric variety into the other.  We show that the toric map induced by the linear dual of the embedding induces a birational morphism between the mirror Calabi-Yau families, and in one case show that it can be extended to the full mirror transition between mirror families. \end{abstract}
\section{Introduction}
Suppose that $X$ and $Y$ are smooth Calabi-Yau threefolds, which we will also take to be projective algebraic varieties over $\mathbb{C}$.  We say that $X$ and $Y$ are related by a {\it conifold transition} if there exists a degeneration of $X$ to a singular variety $X_0$ with a finite number of ordinary double points (also called nodes), and $Y$ can be obtained from $X_0$ by a small resolution of singularities.  A small resolution replaces each node with a projective line so that the resolved variety has a trivial canonical bundle.  There is a conjecture that all pairs of Calabi-Yau threefolds can be obtained from one another by a finite number of such transitions; see the paper \cite{reid}.  More generally, one can consider the idea of an {\it extremal transition} of Calabi-Yau threefolds, defined in \cite{morrison}, where $Y$ is still obtained from $X$ by degenerating and resolving, but the singularities on $X_0$ may be much worse than conifold singularities.  For more background about conifold transitions, and other types of transitions between Calabi-Yau threefolds, see the paper \cite{rossi}.

A node is a singularity which is locally isomorphic, up to a complex analytic change of coordinates, to the singularity of the affine variety $A \subseteq$ Spec $\mathbb{C}[x,y,z,w]$ defined by $xy-zw = 0$.  A small resolution of an ordinary double point replaces the singularity with a $\mathbb{P}^1$, and is locally isomorphic, again up to a complex analytic change of coordinates, to the resolution of the singularity of $A$ given by blowing up the ambient space $\mathbb{C}^4$ on the plane $x = z = 0$.  

The paper \cite{morrison} proposes an elegant connection between extremal transitions and the concept of mirror symmetry introduced by physicists studying string theory.  If $X$ and $Y$ are related by an extremal transition, and $X^*$ and $Y^*$ are their mirrors, then Morrison proposes that $X^*$ and $Y^*$ are also related by an extremal transition, but with the degeneration and resolution switched.  In other words, there should be a degeneration of $Y^*$ to a singular variety $Y^*_0$, and $X^*$ should be a resolution of $Y^*_0$.  In this paper, we study a particular example of a conifold transition and propose that the ``mirror'' conifold transition can be realized by a toric map between the ambient toric varieties of the mirror Calabi-Yau threefolds as provided by the Batyrev-Borisov construction.  The Batyrev-Borisov construction, described in \cite{BB}, is the fundamental method for constructing mirrors of Calabi-Yau manifolds which are complete intersections in Gorenstein toric Fano varieties.  

The relationship between extremal transitions and the Batyrev-Borisov construction has also been explored by Mavlyutov in \cite{mav}, where he gives a construction based on deforming a Calabi-Yau hypersurface in a toric variety by embedding the toric variety as a complete intersection in a higher dimensional toric variety.  This is exactly how the deformations arise in both of the examples in this paper, although Mavlyutov's methods seem to require Minkowski sum decompositions of reflexive polytopes that do not exist in our situations.   

In \cite{morrison}, Morrison himself discusses the relationship between extremal transitions and Batyrev's construction in \cite{Ba}, which is really just the Batyrev-Borisov construction in the case of hypersurfaces.  One may consider a simple inclusion $P \subseteq Q$ of reflexive polytopes as an extremal transition between Calabi-Yau hypersurfaces in the toric varieties corresponding to $P$ and $Q$.  According to Batyrev's construction, the mirror hypersurfaces will lie in the toric varieties corresponding to the dual polytopes $P^*$ and $Q^*$, which will obey the reverse inclusion $Q^* \subseteq P^*$.  This induces an extremal transition between the mirror hypersurfaces.  However, this simpler technique cannot be used in the examples we will look at, since we will deal with a Calabi-Yau manifold which is a codimension two complete intersection in a toric variety, rather than a hypersurface.

Conifold transitions were used by Batyrev et al. in \cite{BCKS} to construct mirrors of complete intersection Calabi-Yau threefolds in Grassmannians.  The main example in this paper is related to the case of quartic hypersurfaces in the Grassmannian $G(2,4)$.  The approach of Batyrev et al.~is to degenerate the Grassmannian $G(k,n)$ to a Gorenstein toric Fano variety called $P(k,n)$.  After the degeneration, a family of complete intersection Calabi-Yau threefolds in $G(k,n)$, $\mathcal{X}$, will degenerate to a family of complete intersections $\mathcal{X}_0$ in $P(k,n)$ where the general member of the family has isolated nodes as singularities.  A so-called MPCP resolution of $P(k,n)$ (as defined in \cite{Ba}) gives a small resolution of general members of $\mathcal{X}_0$, thus giving a nonsingular family $\widehat{\mathcal{X}}_0$ which is related by a conifold transition to the original family $\mathcal{X}$.  Since $\widehat{\mathcal{X}}_0$  consists of MPCP resolutions of complete intersections in the Gorenstein toric Fano variety $P(k,n)$, the Batyrev-Borisov construction can be applied directly to the family $\widehat{\mathcal{X}}_0$ to yield the mirror family $\widehat{\mathcal{X}}^{*}_0$.  By Morrison's philosophy, the mirror family to the original family $\mathcal{X}$ should be produced by a mirror conifold transition of $\widehat{\mathcal{X}}^{*}_0$.  That is, we should obtain the mirror family to $\mathcal{X}$ after an appropriate degeneration of $\widehat{\mathcal{X}}^{*}_0$ to a family where the general member has nodes, followed by a small resolution. 

In \cite{BCKS} this degeneration is accomplished by imposing a relation on the coefficients in the equations describing the family, so the degenerate family is actually a subfamily of $\widehat{\mathcal{X}}^{*}_0$.  In the case of quartic hypersurfaces in $G(2,4)$, $\widehat{\mathcal{X}}^*_0$ is contained in an MPCP resolution of a toric variety $X(\Delta^*_{P(2,4)})$ (the toric variety given by taking cones over the Newton polytope of quartics on $P(2,4)$, which is reflexive).  We will refer to the degenerate nodal subfamily of $\widehat{\mathcal{X}}^{*}_0$ in the $G(2,4)$ case, which consists of hypersurfaces in an MPCP resolution of $X(\Delta^*_{P(2,4)})$, as $\mathcal{X}^*_C$.

Quartic hypersurfaces in $G(2,4)$ are also smooth members of the family $\mathcal{X}_{(2,4)}$ of $(2,4)$ complete intersections in $\mathbb{P}^5$, since $G(2,4)$ is defined by a single quadratic equation in $\mathbb{P}^5$.  Being complete intersections in the Gorenstein toric Fano variety $\mathbb{P}^5$, the Batyrev-Borisov construction also provides a way to find the mirror of this family of Calabi-Yau threefolds.  Thus there are two possible methods that may be used to find the mirror family of quartic hypersurfaces in $G(2,4)$.  Let us denote by $\mathcal{X}^*_{BB}$ the mirror family given by the Batyrev-Borisov construction. 

We start by establishing a birational morphism between the families $\mathcal{X}^*_C$ and $\mathcal{X}^*_{BB}$.  This birational morphism is realized by a toric (monomial) map between the open tori of the toric varieties containing the families.  Furthermore, as a linear map between fans, this map is just the linear dual to the toric embedding of $P(2,4)$ into $\mathbb{P}^5$ as the zero locus of a quadratic binomial equation.  

This birational morphism will only define a map between open subsets of members of the two Calabi-Yau families, so to remedy this, we show that the map can be extended to a map between compact toric varieties containing the compactified and resolved families.  Technically, we show that the map between open tori can be extended to a toric morphism between properly chosen MPCP resolutions of the ambient toric varieties.  This map gives an explicit toric isomorphism between the two different mirror constructions for the family of quartic hypersurfaces in $G(2,4)$, since in \cite{BCKS} the mirror family is given by small resolution of $\mathcal{X}^*_C$.

In the last section, we look at another example to show that the same method can give some results in other cases.




\section{Notations and conventions}

Throughout the paper, we will set $M' = \mathbb{Z}^5$, $M = \mathbb{Z}^4$, $M_\mathbb{R} = M \otimes \mathbb{R}$, and $M'_\mathbb{R} = M' \otimes \mathbb{R}$.  Also define the dual spaces $N = Hom(M, \mathbb{Z})$, $N' = Hom(M', \mathbb{Z})$, $N_\mathbb{R} = N \otimes \mathbb{R}$, and $N'_\mathbb{R} = N' \otimes \mathbb{R}$.

We will denote the convex hull of the sets $S_1, \dots, S_n$ in a real vector space by $Conv(S_1, \dots, S_n)$.  If $v_1, \dots, v_n$ are vectors in a real vector space, then we define the cone over $v_1, \dots, v_n$ as 
$$Cone(v_1, \dots, v_n) = \{r_1 v_1 + \cdots + r_n v_n \ | \ r_1, \dots, r_n \in \mathbb{R}_{\geq 0} \}.$$

If $P$ is a compact convex polytope in a real vector space $V$ with the origin in its interior, we will use $\Sigma(P)$ to denote the fan consisting of cones over the proper faces of $P$. We will also sometimes use $X(P)$ to denote the toric variety associated to $\Sigma(P)$.  If $\Sigma$ is a fan, then $X(\Sigma)$ will denote the toric variety associated to $\Sigma$.  The dual polytope of $P$, $P^*$, is contained in the dual vector space $V^*$ and defined as 
$$\{u \in V^* \ | \ \langle p, u \rangle \geq -1, \forall p \in P \}.$$
Here, $\langle , \rangle$ denotes the natural real-valued pairing between $V$ and $V^{*}$.

We define the Newton polytope of a real-valued function $\varphi$ on a real vector space $V$ by $$Newt(\varphi) = \{ u \in V^{*} \ | \ \langle u,v \rangle \geq -\varphi(v), \forall v \in V \}.$$  If $\Sigma$ is a complete fan and $\varphi : \Sigma \rightarrow \mathbb{R}$ is an integral piecewise linear function on $\Sigma$, then $\varphi$ defines a line bundle $\mathcal{L}(\varphi)$ on $X(\Sigma)$, and $Newt(\varphi)$ is the Newton polytope of global sections of $\mathcal{L}(\varphi)$.  Lattice points in the Newton polytope correspond to monomial global sections of $\mathcal{L}(\varphi)$.

All convex functions in this paper will be lower convex, meaning that $$\varphi(ax+by) \leq a \varphi(x) + b \varphi(y)$$ for any $x, y \in V$ and $a, b \in \mathbb{R}$ with $0 \leq a, b \leq 1$, $a+b=1$.  With our conventions, ample line bundles on a toric variety $X(\Sigma)$ are associated to strictly lower convex piecewise linear functions $\varphi : \Sigma \rightarrow \mathbb{R}$.  (Note that some toric geometry books such as \cite{CLS} use the opposite convention, so for them ample line bundles would be associated to strictly upper convex functions!)

A reflexive polytope $\Delta \subseteq V$ is a lattice polytope with the origin in its interior, whose dual polytope $\Delta^* \subseteq V$ is also a lattice polytope.  In this paper we will often use the concept of a ``crepant resolution'' or ``crepant subdivision'' of the fan $\Sigma(\Delta)$.  A crepant subdivision of $\Sigma(\Delta)$ is another complete fan $\widehat{\Sigma}(\Delta)$ which is a subdivision of $\Sigma(\Delta)$, and whose rays are all rays over lattice points in $\Delta$. 

\section{Review of mirror construction via conifold transitions}

The Pl\" ucker embedding of $G(2,4)$ in $\mathbb{P}^5$ is defined by the single quadratic equation $$z_0z_1-z_2z_3+z_4z_5=0$$ where $z_0, \dots, z_5$ are homogeneous coordinates on $\mathbb{P}^5$.  Intersecting $G(2,4)$ with a generic degree four hypersurface in $\mathbb{P}^5$ gives a smooth Calabi-Yau threefold which is part of the family of $(2,4)$ complete intersections in $\mathbb{P}^5$.  

The method used by \cite{BCKS} to construct the mirror of this family of threefolds starts by degenerating $G(2,4)$ to the toric variety $P(2,4)$ via the equation $$tz_0 z_1-z_2 z_3+z_4z_5=0$$ in $\mathbb{P}^5 \times \mathbb{C}$.  Fixing a value of $t$ with $t \neq 0$ results in a variety isomorphic to $G(2,4)$.  When $t=0$ we get the singular toric variety $P(2,4)$, defined by the equation $z_2z_3=z_4z_5$.  The singular locus of $P(2,4)$ consists of the line of nodes $z_2=z_3=z_4=z_5=0$, and a generic degree four hypersurface in $\mathbb{P}^5$ will intersect the line transversely in four points, resulting in a Calabi-Yau variety with four nodes as singularities.  

There exist toric resolutions of $P(2,4)$ to a nonsingular toric variety $\widehat{P}(2,4)$.  The toric variety $\widehat{P}(2,4)$ can be obtained by blowing up $P(2,4)$ on a toric divisor of the form $z_i=z_j=0$ for any choice of $i \in \{ 2,3 \}$, $j \in \{4,5 \}$.  Such a resolution induces a small resolution of the four nodes on a generic Calabi-Yau hypersurface in $P(2,4)$.  The degeneration of a Calabi-Yau hypersurface $X$ in $G(2,4)$ to a hypersurface in $P(2,4)$, followed by this small resolution, is the conifold transition used by \cite{BCKS} to construct the mirror of $X$.

To perform the mirror conifold transition, \cite{BCKS} starts with the mirror family to the family $\mathcal{X}_{\widehat{P}(2,4)}$ of Calabi-Yau hypersurfaces in $\widehat{P}(2,4)$ as provided by Batyrev's construction in \cite{Ba}.  The fan for $P(2,4)$ may be described as the cones over faces of a reflexive polytope $\Delta_{P(2,4)} \subseteq  M_\mathbb{R} \cong \mathbb{R}^4$ with six vertices, defined as $$\Delta_{P(2,4)} = Conv(u_1,u_2,u_3, u_4, u_5, u_6)$$ where the vertices $u_i$ are defined as 
$$u_1 = f_1, \ u_2 = f_2, \ u_3 = f_3,$$
$$u_4 = f_4, \ u_5 = -f_1-f_2-f_3, \ u_6 = f_1+f_2-f_4,$$
with $f_1, \dots, f_4$ being the standard basis of $M_\mathbb{R}$.
According to Batyrev's construction, the mirror family to $\mathcal{X}_{\widehat{P}(2,4)}$ is given by a Calabi-Yau compactification of the subvariety of the torus $T = \hbox{Spec}\ \mathbb{C}[M]$ defined by $$-1+a_1X_1+a_2X_2+a_3X_3+a_4X_4+a_5(X_1X_2X_3)^{-1}+a_6X_1X_2X^{-1}_4= 0$$ where the $a_i$ are generic coefficients and $X_i = z^{f_i}$.  The degenerate subfamily of this family is given by the additional equation $a_1a_2 = a_4a_6$.  This equation arises from the fact that $X_1$, $X_2$, $X_4$ and $X_1X_2X^{-1}_4$ are the monomials corresponding to the vertices of the square face $Conv(u_1, u_2, u_4, u_6)$ of $\Delta_{P(2,4)}$, with $\{ u_1, u_2 \}$ and $\{ u_4, u_6 \}$ being pairs of opposite corners.  The toric stratum of $P(2,4)$ corresponding to the cone over this face is the singular locus of $P(2,4)$ and the location of the nodes of a generic Calabi-Yau quartic hypersurface in $P(2,4)$.

By Batyrev's construction, a Calabi-Yau compactification of this family exists in the toric variety given by an MPCP resolution of the toric variety $X(\Delta_{P(2,4)}^{*})$, which is defined by the fan of cones over the faces of the dual polytope $\Delta_{P(2,4)}^{*} \subseteq N_\mathbb{R}$.  The face of $\Delta_{P(2,4)}^{*}$ that is dual to the square $Conv(u_1, u_2, u_4, u_6)$ is the line segment $$L = Conv((-1,-1,-1,-1),(-1,-1,3,-1)),$$ using the basis for $N_\mathbb{R}$ which is dual to $f_1,\dots, f_4$.  After an MPCP resolution of $X(\Delta_{P(2,4)}^{*})$, $L$ will be subdivided into four smaller line segments.  The open subset of the MPCP resolution corresponding to the union of the cones over these line segments will contain four nodes of a generic Calabi-Yau hypersurface in the degenerate subfamily.  A small resolution of these nodes yields the mirror of the original Calabi-Yau manifold $X$.    

As in \cite{BCKS}, using the $(\mathbb{C}^{*})^4$ action on $T$ we may assume without loss of generality that $a_1=a_2=a_3=a_4=1$, so that the family becomes
\begin{equation} \label{cfoldfamily}
-1+X_1+X_2+X_3+X_4+a_5(X_1X_2X_3)^{-1}+a_6X_1X_2X^{-1}_4= 0
\end{equation}
and the degenerate subfamily is defined by $a_6 =1$.  We will refer to this degenerate family as $\mathcal{X}^{*}_{C}$.  

\section{The Batyrev-Borisov mirror and its birational equivalence}

Now we establish a birational morphism between the family $\mathcal{X}^{*}_{C}$ and the mirror family $\mathcal{X}^*_{BB}$ to $(2,4)$ hypersurfaces in $\mathbb{P}^5$ as given by the Batyrev-Borisov construction.  

Let $e_1, \dots, e_5$ be the standard basis for $M'_\mathbb{R} \cong \mathbb{R}^5$ and let $\Delta \subseteq M'_\mathbb{R}$ be the reflexive polytope $Conv(e_1,\dots, e_5, -e_1-e_2-e_3-e_4-e_5)$.  Then the fan $\Sigma(\Delta)$ consisting of cones over the proper faces of $\Delta$ is the fan for $\mathbb{P}^5$.  Define the piecewise linear functions $\varphi_1$, $\varphi_2$ on $\Sigma(\Delta)$ by $\varphi_1(e_i) = 1$ for $i = 1$, 2, 3, 4, and $\varphi_1(e_5) = \varphi_1(-e_1-\cdots-e_5) = 0$, and $\varphi_2(e_i) = 0$ for $i = 1$, 2, 3, 4, and $\varphi_2(e_5) = \varphi_2(-e_1-\cdots-e_5) = 1$.  The line bundles on $\mathbb{P}^5$ associated to $\varphi_1$ and $\varphi_2$ are $\mathcal{O}_{\mathbb{P}^5}(4)$ and $\mathcal{O}_{\mathbb{P}^5}(2)$ respectively.

The Batyrev-Borisov construction now states that the mirror family of $(2,4)$ complete intersections in $\mathbb{P}^5$ is a Calabi-Yau compactification of the subvariety of $T'= \hbox{Spec} \ \mathbb{C}[M']$
defined by 
\begin{align*}
b_1 Y_1 + b_2 Y_2 + b_3 Y_3 + b_4 Y_4 &= 1 \\
b_5 Y_5 + b_6 (Y_1Y_2Y_3Y_4Y_5)^{-1} &= 1
\end{align*}
with the $b_i$ generic coefficients and $Y_i = z^{e_i}$.  By the $(\mathbb{C}^{*})^5$ action on $T'$ we may assume without loss that $b_1=b_2=b_3=b_5=b_6=1$, so the equations reduce to 
\begin{align*}
Y_1+Y_2+Y_3+b_4 Y_4 &= 1 \\
Y_5 + (Y_1Y_2Y_3Y_4Y_5)^{-1} &= 1.
\end{align*}
We will refer to the second family as $\mathcal{X}^*_{BB}$.

In the following theorem, we will define a birational equivalence between the families $\mathcal{X}^*_{BB}$ and $\mathcal{X}^*_{C}$ by defining a map between the tori $T'$ and $T$.  These families should properly be considered to lie in $T' \times \mathbb{C}$ and $T \times \mathbb{C}$, with the $\mathbb{C}$ factor representing the parameters $b_4$ and $a_5$, respectively, and a morphism between the families could be given as a map between $T' \times \mathbb{C}$ and $T \times \mathbb{C}$ which maps $\mathcal{X}^*_{BB}$ to $\mathcal{X}^*_{C}$ and respects projection to $\mathbb{C}$.  However, in this case, it is possible to realize the map as a single map from $T'$ to $T$; in other words, the map does not depend on the parameter $b_4$.  

\begin{thm} \label{birational1} After setting $b_4 = a_5$, the families $\mathcal{X}^*_{BB}$ and $\mathcal{X}^{*}_C$ are birationally equivalent via the map $f : T' \rightarrow T$ defined by $X_1 = (Y_2Y_3Y_4Y_5)^{-1}$, $X_2 = Y_2Y_5$, $X_3 = Y_3$, $X_4 = (Y_1Y_3Y_4Y_5)^{-1}$. \end{thm}

\begin{proof} First we observe that the defining equation for $\mathcal{X}^{*}_C$, equation \eqref{cfoldfamily}, may be factored as 
$$-1+(1+X_4X^{-1}_2)(X^{-1}_4X_1X_2+X_2)+X_3+a_5(X_1X_2X_3)^{-1}=0.$$
After doing this, it is straightforward to check that the rational map from $T$ to $T'$ defined by
\begin{align*}
Y_1 &= (1+X_4X^{-1}_2)(X^{-1}_4X_1X_2) \\
Y_2 &= (1+X_4X^{-1}_2)X_2 \\ 
Y_3 &= X_3 \\
Y_4 &= (X_1X_2X_3)^{-1} \\
Y_5 &= (1+X_4X^{-1}_2)^{-1} 
\end{align*}
when restricted to any member of $\mathcal{X}^{*}_C$, is a birational inverse for the restriction of $f$ to the corresponding member of $\mathcal{X}^*_{BB}$.\end{proof}

\section{Extension to the ambient toric varieties}

Because the map $f : T' \rightarrow T$ is defined by monomials, it is associated to a $\mathbb{Z}$-linear map $g : M \rightarrow M'$ and its dual map $g^{*} : N' \rightarrow N$. In the standard bases of $M$ and $M'$, $g$ is given by the matrix
\[ \left( \begin{array}{cccc}
0 & 0 & 0 & -1 \\
-1 & 1 & 0 & 0 \\
-1 & 0 & 1 & -1 \\
-1 & 0 & 0 & -1 \\ 
-1 & 1 & 0 & -1 \\
\end{array} \right)\]
acting on the left (with elements of $M$ and $M'$ written as column vectors), and in the dual bases of $N$ and $N'$, $g^*$ is defined by its transpose.  The kernel of $g^{*}$ is easily seen to be the $\mathbb{Z}$-span of the vector $(1,1,0,0,-1)$. We will also use $g$ and $g^*$ to denote the extensions of these $\mathbb{Z}$-linear maps to maps from $M_\mathbb{R}$ to $M'_\mathbb{R}$ and $N'_\mathbb{R}$ to $N_\mathbb{R}$. 

\begin{rmk} Although our focus for the rest of the paper will be $g^*$ acting as a map between the toric varieties on the mirror side, one can check that the map $g : M_\mathbb{R} \rightarrow M'_\mathbb{R}$ defines a map of fans $g: \Sigma(\Delta_{P(2,4)}) \rightarrow \Sigma(\Delta)$ which embeds $P(2,4)$ into $\mathbb{P}^5$ as the zero locus of the equation $$z^{e^*_1+e^*_2-e^*_5} = z^0.$$  Here $e^*_1, \dots, e^*_5$ is the dual basis to the standard basis $e_1, \dots, e_5$ of $M'_\mathbb{R}$.  Because $e^*_1+e^*_2-e^*_5$ and $0$ are both contained in $Newt(\varphi_2)$, we may consider both sides of the equation to lie in $\mathcal{O}_{\mathbb{P}^5}(2)$, with $z^{e^*_1+e^*_2-e^*_5}$ the quadratic monomial vanishing on the divisors of $\mathbb{P}^5$ corresponding to the vertices $e_1$ and $e_2$ of $\Delta$, and $z^0$ the monomial vanishing on the divisors corresponding to $e_5$ and $-e_1-e_2-e_3-e_4-e_5$.  (Also note that $e^*_1+e^*_2-e^*_5$ generates the kernel of $g^*$.)\end{rmk} 

Once the family $\mathcal{X}^*_{BB}$ is compactified and resolved by the Batyrev-Borisov construction, it will lie in an MPCP resolution of a certain Gorenstein toric Fano variety with a fan consisting of cones over a polytope $\nabla \subseteq N'_\mathbb{R}$.  With $\varphi_1, \varphi_2 : M'_\mathbb{R} \rightarrow \mathbb{R}$ defined as before, set $\nabla_1 = Newt(\varphi_1)$ and $\nabla_2 = Newt(\varphi_2)$.  Then the Batyrev-Borisov construction in \cite{BB} says that $\nabla = Conv(\nabla_1, \nabla_2)$.

One may verify that in the basis of $N'_\mathbb{R}$ that is dual to the standard basis of $M'_\mathbb{R} \cong \mathbb{R}^5$, the vertices of $\nabla_1$ are the rows of the matrix 
\[ \left( \begin{array}{ccccc}
-1 & -1 & -1 & -1 & 0\\
3 & -1 & -1 & -1 & 0 \\
-1 & 3 & -1 & -1 & 0 \\
-1 & -1 & 3 & -1 & 0 \\ 
-1 & -1 & -1 & 3 & 0 \\
-1 & -1 & -1 & -1 & 4
\end{array} \right)\]
and the vertices of $\nabla_2$ are the rows of 
\[ \left( \begin{array}{ccccc}
0 & 0 & 0 & 0 & -1 \\
2 & 0 & 0 & 0 & -1 \\
0 & 2 & 0 & 0 & -1 \\
0 & 0 & 2 & 0 & -1 \\ 
0 & 0 & 0 & 2 & -1 \\
0 & 0 & 0 & 0 & 1 
\end{array} \right).\]
The vertices of $\nabla$ consist of the twelve rows of the above two matrices.  

We will need the following convex geometry result:

\begin{prop} \label{newt} Let $p : X \rightarrow Y$ be a linear map between finite dimensional real vector spaces, and let $\varphi$ be a real-valued lower convex function on $Y$.  Let $p^{*} : Y^{*} \rightarrow X^{*}$ be the dual map.  Then $Newt(\varphi \circ p) = p^{*}(Newt(\varphi))$.  \end{prop}

\begin{proof} First we show the containment $p^{*}(Newt(\varphi)) \subseteq Newt(\varphi \circ p)$.  We must show that for any $y \in Newt(\varphi)$, that $\langle p^{*}(y), z \rangle \geq -\varphi(p(z))$ for all $z \in X$.  This is true because $\langle p^{*}(y), z \rangle = \langle y, p(z) \rangle$, and $\langle y, p(z) \rangle \geq -\varphi(p(z))$ because $y \in Newt(\varphi)$.   

To show the reverse containment, let $x \in Newt(\varphi \circ p)$.  We first claim that because $x$ is such that $\langle x, z \rangle \geq -\varphi(p(z))$ for all $z \in X$, then $x \in p^*(Y^{*})$.  The latter condition is equivalent to $x$, regarded as a function on $X$, being constant on the fibers of $p$, which follows from the first condition because $-\varphi(p(z))$ is constant on the fibers of $p$ and $x$ is linear.  Therefore any $x$ satisfying the first condition descends to a function $x'$ on $p(X) \subseteq Y$, and is completely determined by its values on this subspace.  We choose a complementary subspace $S \subseteq Y$ to $p(X)$ and choose a linear function $L$ on $S$ such that $L(y) \geq -\varphi(y)$ for all $y \in S$, which is possible since $\varphi$ is lower convex.  Let $L'$ be the unique linear function such that $L'|_S = L$ and $L'|_{p(X)} = x'$.  Then by convexity of $\varphi$, $L'(y) \geq -\varphi(y)$ for all $y \in Y$.  Regarding $L'$ as an element of $Y^{*}$, this means that $L' \in Newt(\varphi)$, and we have that $p^{*}(L') = x$ since $L'|_{p(X)} = x'$, which proves the containment.
\end{proof}

\begin{lemma} \label{imagething} With $g^{*} : N'_\mathbb{R} \rightarrow N_\mathbb{R}$ defined as above, $g^{*}(\nabla) = \Delta^{*}_{P(2,4)}.$ \end{lemma}
  
\begin{proof} First we apply Proposition \ref{newt} with $\varphi = \varphi_1$ and $p = g$.  It is straightforward to check that $\varphi_1 \circ g$ is the piecewise linear function $h_1$ on $\Sigma(\Delta_{P(2,4)})$ associated to the anticanonical bundle of $P(2,4)$, in other words, the piecewise linear function which is equal to 1 on the primitive integral generators of all rays in the fan, $u_i$ for $1 \leq i \leq 6$.  Since $\Delta^{*}_{P(2,4)} = Newt(h_1)$, the proposition gives us that $g^{*}(Newt(\varphi_1)) = \Delta^{*}_{P(2,4)}$.

One can also check that $\varphi_2 \circ g$ gives a piecewise linear function $h_2$ on $\Sigma(\Delta_{P(2,4)})$ which is equal to one on $u_1$, $u_2$, $u_4$ and $u_6$ (the vertices of the square face of $\Delta_{P(2,4)}$), and zero on all other rays.  Because $h_2 \leq h_1$, $$g^{*}(Newt(\varphi_2)) = Newt(h_2) \subseteq Newt(h_1) = \Delta^{*}_{P(2,4)}.$$
Thus we have that
\begin{gather*}
g^{*}(\nabla) = Conv(g^{*}(Newt(\varphi_1)), g^{*}(Newt(\varphi_2))) = \\
Conv(g^{*}(Newt(\varphi_1)) = \Delta^{*}_{P(2,4)}.
\end{gather*} \end{proof}

Ultimately, we would like to extend $g^*$ to a toric morphism between compact toric varieties containing the (partially) resolved Calabi-Yau families $\mathcal{X}^*_C$ and $\mathcal{X}^*_{BB}$, so that $g^*$ resolves a general member of $\mathcal{X}^*_C$ to a smooth member of $\mathcal{X}^*_{BB}$.  To do this, we need to deal with the fact that these compactified and resolved families live in toric varieties which are ``MPCP resolutions'' as defined by Batyrev in \cite{Ba}.  

Given a reflexive polytope $\Delta$, an MPCP resolution of $X(\Delta)$ is defined by a fan $\widehat{\Sigma}(\Delta)$ which is a crepant subdivision of $\Sigma(\Delta)$.  To be an MPCP resolution, $\widehat{\Sigma}(\Delta)$ must be complete, simplicial, and have a set of rays equal to the set of rays over nonzero lattice points in $\Delta$.  If a reflexive polytope contains a large number of lattice points, as in the present case, then dealing with its MPCP resolutions directly can be difficult because the resolutions are combinatorially very complicated.  We will approach this problem indirectly, by showing that an MPCP resolution $\widehat{\Sigma}(\Delta^*_{P(2,4)})$ satisfying certain conditions can be lifted to an MPCP resolution $\widehat{\Sigma}(\nabla)$, so that the toric morphism $g^* : \widehat{\Sigma}(\nabla) \rightarrow \widehat{\Sigma}(\Delta^*_{P(2,4)})$ exists.  In other words, the image of any cone $C \in \widehat{\Sigma}(\nabla)$ under $g^*$ will be contained in some cone of $\widehat{\Sigma}(\Delta^*_{P(2,4)})$. 

We introduce the following definition:

\begin{definition} \label{conelifting} Let $\Delta \subseteq N_\mathbb{R}$ and $\Delta' \subseteq N'_\mathbb{R}$ be reflexive polytopes, let $\Sigma_1$ be a crepant subdivision of $\Sigma(\Delta)$, and let $L : N'_\mathbb{R} \rightarrow N_\mathbb{R}$ be a $\mathbb{Z}$-linear map.  We say that $(\Delta', \Delta, \Sigma_1, L)$ satisfies the {\it cone lifting property} if the following holds: if $C \subseteq N_\mathbb{R}$ is any cone whose rays are rays over lattice points in $\Delta$, and which is contained in some cone $C'$ of $\Sigma_1$, then the intersection fan $L^{-1}(C) \cap \Sigma(\Delta')$ in $N'_\mathbb{R}$ consists of cones whose rays are rays over lattice points in $\Delta'$.  (The intersection fan $L^{-1}(C) \cap \Sigma(\Delta')$ is defined as the fan consisting of the cones $L^{-1}(C) \cap D$ where $D$ is any cone of $\Sigma(\Delta')$.) \end{definition}

The cone lifting property will be used to solve the problem of extending the toric morphism $g^*$, via the following lemma:

\begin{lemma} \label{cliftlemma} Suppose that $(\Delta', \Delta, \Sigma_1, L)$ satisfies the cone lifting property.  Let $\widehat{\Sigma}(\Delta)$ be an MPCP subdivision of $\Sigma(\Delta)$ which is also a subdivision of $\Sigma_1$.  Then there exists an MPCP subdivision $\widehat{\Sigma}(\Delta')$ of $\Sigma(\Delta')$ such that the morphism of fans $L : \widehat{\Sigma}(\Delta') \rightarrow \widehat{\Sigma}(\Delta)$ exists. \end{lemma}

\begin{proof} Let $\Sigma_{int}$ be the intersection fan consisting of all cones $L^{-1}(C) \cap D$, where $C$ is any cone of $\widehat{\Sigma}(\Delta)$ and $D$ is any cone of $\Sigma(\Delta')$.  Then $\Sigma_{int}$ will be a projective subdivision of $\Sigma(\Delta')$, and it will be crepant because of the cone lifting property.  If $\Sigma_{int}$ is not already a MPCP subdivision of $\Sigma(\Delta')$, then it can be further subdivided into an MPCP subdivision $\widehat{\Sigma}(\Delta')$.  By construction, the image of every cone of $\widehat{\Sigma}(\Delta')$ under $L$ will be contained in some cone of $\widehat{\Sigma}(\Delta)$. 
\end{proof}

\begin{prop} \label{cliftcriteria} With the same notation as Definition \ref{conelifting}, suppose that $(\Delta', \Delta, \Sigma_1, L)$ are such that:

1. For each cone $D \in \Sigma(\Delta')$, $L(D)$ is a union of cones in $\Sigma_1$.

2. Considering $L : N' \rightarrow N$ as a $\mathbb{Z}$-linear map, the kernel of $L$ is one-dimensional and generated over $\mathbb{Z}$ by $v \in N'$, with both $v$ and $-v$ contained in $\Delta'$.

3. For each lattice point $\ell \in \Delta$, there is a lattice point $\ell' \in \Delta'$ with $L(\ell') = \ell$.

Then $(\Delta', \Delta, \Sigma_1, L)$ satisfies the cone lifting property. \end{prop}

\begin{proof} According to Definition \ref{conelifting}, we need to show that for any cone $C$ which is generated by lattice points in $\Delta$ and contained in some cone $C'$ of $\Sigma_1$, the rays of the intersection fan $L^{-1}(C) \cap \Sigma(\Delta')$ are rays over lattice points of $\Delta'$.  

Let $r$ be a ray of the intersection fan $L^{-1}(C) \cap \Sigma(\Delta')$.  Then by definition, $r = L^{-1}(E) \cap D$, where $D$ is some cone of $\Sigma(\Delta')$ and $E$ is some face of $C$.  (All the faces of $L^{-1}(C)$ must be inverse images of faces of $C$.)  Now we claim that either $r$ is the ray over $v$ or $-v$, or $L$ maps $r$ bijectively onto $E \cap L(D)$.  If $r$ is the ray over $v$ or $-v$, then $r$ is a ray over a lattice point in $\Delta'$ and we have nothing to prove.  Otherwise, if $p \in E \cap L(D)$, then there must be some $p' \in D$ such that $L(p') = p$.  Since $p \in E$, we have that $p' \in L^{-1}(E)$ also.  Thus $p' \in r = L^{-1}(E) \cap D$ and $L$ is surjective from $r$ onto $E \cap L(D)$; since $r$ is one-dimensional, it must also be injective.  

Now by assumption, $C$ and therefore $E$ is contained in a cone of $\Sigma_1$, say $C_1$.  Since $L(D)$ is a union of cones in $\Sigma_1$, $L(D) \cap C_1$ must be a (not necessarily proper) face of $C_1$, say $F_1 \subseteq C_1$.  Then $E \cap L(D) = E \cap F_1$.  In general, if a polyhedral cone $A_1$ is contained in another polyhedral cone $A_2$ and we intersect $A_1$ with a face of $A_2$, we must get a face of $A_1$.  So because $E \subseteq C_1$ and $F_1$ is a face of $C_1$, $E \cap L(D) = E \cap F_1$ must be a face of $E$.  But $E \cap L(D)$ is also ray from the previous paragraph, and since $E$ is a face of $C$, the ray must be a ray $R$ of $C$, which must be a ray over a lattice point in $\Delta$.  

Thus the ray $r$ is a ray of the two-dimensional intersection fan $L^{-1}(R) \cap \Sigma(\Delta')$.  Let $\ell \in \Delta$ be the primitive integral generator of $R$, and let $\ell' \in \Delta'$ be a lattice point such that $L(\ell') = \ell$, which exists because of condition 3 in the Proposition statement.

The support of the intersection fan $L^{-1}(R) \cap \Sigma(\Delta')$, which is just $L^{-1}(R)$, can be expressed as the union $$ L^{-1}(R) = Cone(\ell',v) \cup Cone(\ell', -v).$$  By Proposition 3.1 from \cite{nill}, $Cone(\ell', v) \cap \Delta'$ will consist of $Conv(\ell',v,0)$, or of $Conv(z,\ell',v,0)$ for some $z$ which is an integer linear combination of $\ell'$ and $v$, and in particular a lattice point, and $z \in \Delta'$.  (This follows from the fact that since $\ell' \neq -v$ by assumption, we must be in case 1 or 3 of Proposition 3.1.)  Likewise, $Cone(\ell',-v) \cap \Delta$ will consist either of $Conv(\ell',-v,0)$ or $Conv(w,\ell',-v,0)$ for some lattice point $w \in \Delta$.  See Figure \ref{fanpic} for an illustration of a possible case.  With any of these possibilities, we can conclude that all the rays of the intersection fan $L^{-1}(R) \cap \Sigma(\Delta')$ are rays over lattice points in $\Delta$. \end{proof}  

\begin{figure}[h] 
  \begin{center} \scalebox{0.75}{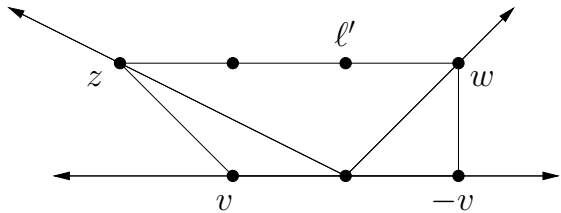} \end{center}
   \caption{\label{fanpic} Illustration of a possible configuration of the fan $L^{-1}(R) \cap \Sigma(\Delta')$} 
\end{figure}

We want to obtain a fan $\Sigma_1$, which is a crepant subdivision of $\Sigma(\Delta^*_{P(2,4)})$, such that $(\nabla, \Delta^*_{P(2,4)}, \Sigma_1, g^*)$ satisfies the conditions of Proposition \ref{cliftcriteria} and therefore satisfies the cone lifting property.  We need $\Sigma_1$ to satisfy the condition that for every cone $C \in \Sigma(\nabla)$, $g^*(C)$ is equal to a union of cones in $\Sigma(\Delta^*_{P(2,4)})$.  The data for such a fan $\Sigma_1$ is given in Table \ref{sigma1}, and the necessary facts can be verified using the Polyhedra package \cite{birkner} for the computer algebra program Macaulay2 \cite{M2}.   

In order to verify condition 1 of Proposition \ref{cliftcriteria}, we took the following approach: given any cone $C \in \Sigma(\nabla)$, $C = Cone(w_1, \dots, w_n)$ with $w_1, \dots, w_n$ vertices of $\nabla$, $g^*(C) = Cone(g^*(w_1),\dots,g^*(w_n))$, a cone of dimension $d$.  To show that $g^*(C)$ is a union of cones of $\Sigma_1$, it suffices to show that the volume of the polytope $V = Conv(0,g^*(w_1),\dots,g^*(w_n))$ is equal to the sum $$\sum_{U \in S} vol (U),$$
where $S$ is the set of polytopes 
\begin{gather*}
\{ U = Conv(0, t_1, \dots, t_m) \ | \ U \subseteq V, \ dim(U) = d, \
 Cone(t_1, \dots, t_m) \in \Sigma_1 \\
\hbox{and} \ t_i \in \{v_1,\dots,v_{13} \} \ \forall 1 \leq i \leq m \}.
\end{gather*}
The volume is taken with respect to any fixed choice of volume form on $Span_\mathbb{R}(g^*(C))$.  The number of cones in $\Sigma(\nabla)$ of any dimension is sufficiently small that this computation can be performed in a few minutes with a Macaulay2 script.  It is also easy to check that the second and third conditions of Proposition \ref{cliftcriteria} hold, using Macaulay2 where necessary.

\begin{table}
\begin{tabular}{ |l |m{9cm} | }
  \hline
 Generators of rays of $\Sigma_1$ & {\begin{gather*} v_1 = (3,-1,-1,3), \ v_2 = (3,-1,-1,-1), \\ 
 v_3 = (-1,3,-1,3), \ v_4 = (-1,-1,3,-1), \\
v_5 = (-1,-1,-1,-1), \ v_6 = (-1,3,-1,-1), \\ 
v_7 = (1,-1,0,1), \ v_8 = (1,-1,0,-1), \\
v_9 = (-1,1,0,1), \ v_{10} = (-1,-1,2,-1), \\ 
v_{11} = (-1,-1,0,-1), \ v_{12} = (-1,1,0,-1), \\ v_{13} = (1,1,-1,1) \end{gather*}} \\
  \hline
  Maximal cones of $\Sigma_1$ & 
   \begin{gather*} \{ v_5,v_6,v_3,v_{13} \}, \ \{ v_5,v_2,v_1,v_{13} \}, \\
\{ v_2,v_4,v_1,v_{13} \}, \ \{ v_6,v_4,v_3,v_{13} \}, \\
\{ v_5,v_2,v_6,v_{13} \}, \ \{ v_2,v_6,v_4,v_{13} \}, \\
\{ v_5,v_1,v_3,v_{13} \}, \ \{ v_4,v_1,v_3,v_{13} \}, \\
\{ v_{11},v_{12},v_{10},v_9 \}, \ \{ v_5,v_6,v_3,v_{11},v_{12},v_9 \}, \\
\{ v_{11},v_8,v_{10},v_7 \}, \ \{ v_5,v_2,v_1,v_{11},v_8,v_7 \}, \\
\{ v_2,v_4,v_1,v_8,v_{10},v_7 \}, \ \{ v_6,v_4,v_3,v_{12},v_{10},v_9 \}, \\
\{ v_{11},v_8,v_{12},v_{10} \}, \ \{ v_2,v_5,v_6,v_{11},v_8,v_{12} \}, \\
\{ v_2,v_6,v_4,v_8,v_{12},v_{10} \}, \ \{ v_{11},v_{10},v_7,v_9 \}, \\
\{ v_5,v_1,v_3,v_{11},v_7,v_9 \}, \ \{ v_4,v_1,v_3,v_{10},v_7,v_9 \} 
   \end{gather*} \\
  \hline
\end{tabular}
\caption{Data for fan $\Sigma_1$ in $N_\mathbb{R}$}
\label{sigma1}
\end{table}

$\Sigma_1$ was constructed by starting with the polyhedron $Q \subseteq N_\mathbb{R} \oplus \mathbb{R}$ defined by
$$Q = \{ (n, r) \ | \ n \in \Delta^*_{P(2,4)}, \ r \geq \varphi(n) \}$$
where $\varphi : N_\mathbb{R} \rightarrow \mathbb{R}$ is the piecewise linear support function which is identically equal to 1 on the boundary of $\Delta^*_{P(2,4)}$.  $Q$ is a polyhedron extending infinitely upwards with two types of faces, unbounded or ``vertical'' faces, and bounded or ``lower'' faces.  (See \cite{GBB}, pp.~20-22, where this type of polyhedron and its dual fan are analyzed further.)  We then take the convex hull of $Q$ with the points $(v_i, 1 - \epsilon)$, where $7 \leq i \leq 13$ and $\epsilon$ is a small rational number, to get a new polytope $Q'$.  The lower faces of $Q'$ are the graph of a piecewise linear function over $\Delta^*_{P(2,4)}$, and this function is strictly convex on the fan $\Sigma_1$.  

Note that $v_7, \dots, v_{12}$ are the images of the vertices of $\nabla_2$ under $g^*$, while $v_{13}$ is the center of the unique two-dimensional square face of $\Delta^*_{P(2,4)}$, $$S_1 = Conv(v_1,v_2,v_3,v_6).$$  The images of faces of $\nabla_1$, which are all simplices, induce the subdivision of $S_1$ in Figure \ref{square}, because a tetrahedral face of $\nabla_1$ is mapped to $S_1$ by flattening.

\begin{figure}[h] 
  \begin{center} \scalebox{0.5}{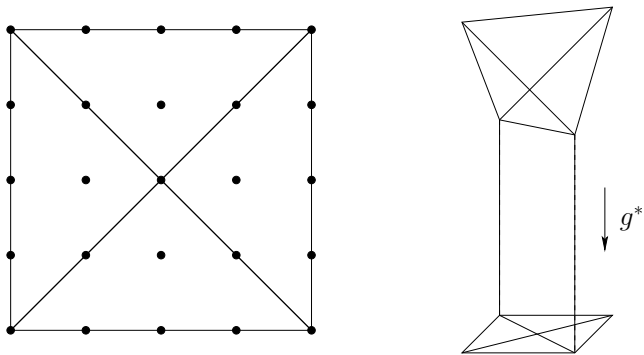} \end{center}
   \caption{\label{square} Subdivision of the square $S_1$ in the fan $\Sigma_1$, which is induced by the linear map $g^*$} 
\end{figure}

By the preceding calculations, it follows that $(\nabla,\Delta^*_{P(2,4)},\Sigma_1, g^*)$ satisfies the cone lifting property, and therefore by Lemma \ref{cliftlemma} we can pick any MPCP resolution $\widehat{\Sigma}(\Delta^*_{P(2,4)})$ which is a subdivision of $\Sigma_1$, and lift it to an MPCP resolution $\widehat{\Sigma}(\nabla)$ such that the map $g^* : \widehat{\Sigma}(\nabla) \rightarrow \widehat{\Sigma}(\Delta^*_{P(2,4)})$ exists.  We can now state our main result:

\begin{thm} There exist MPCP resolutions such that the map of fans $g^*: \widehat{\Sigma}(\nabla) \rightarrow \widehat{\Sigma}(\Delta^*_{P(2,4)})$ exists.  The induced map between toric varieties restricts to a birational morphism $g^* : \mathcal{X}^*_{BB} \rightarrow \mathcal{X}^*_C$ between the compactified and (partially) resolved Calabi-Yau families $\mathcal{X}^*_{BB}$ and $\mathcal{X}^*_C$. \end{thm}

We can ask more about the behavior of $g^*$, in particular, whether it is a resolution of singularities.  Since we know the family $\mathcal{X}^*_{BB}$ is generically smooth, this would mean verifying that $g^*$ is an isomorphism away from the nodal singularities in members of the family $\mathcal{X}^*_C$.  This implied by the following general fact: if $f: X \rightarrow Y$ is a regular birational map of complex projective varieties, $X$ is smooth with trivial canonical bundle, and $Y$ is normal, then $f$ must be an isomorphism onto the nonsingular locus of $Y$.

\section{Another example}

To show that the same methods can yield results in other cases, we present a brief analysis of the example in section 3.3 in \cite{morrison}.  In this example, Morrison describes smoothing a Calabi-Yau hypersurface in the weighted projective space $\mathbb{P}^{(1,1,2,2,2)}$ by embedding $\mathbb{P}^{(1,1,2,2,2)}$ in $\mathbb{P}^5$ torically as the solution of a quadratic binomial equation.  (Indeed, if $z_0, \dots, z_5$ are homogeneous coordinates on $\mathbb{P}^5$, then the variety defined by $z_0 z_1 = z^2_2$ is isomorphic to $\mathbb{P}^{(1,1,2,2,2)}$.)  If the quadratic binomial equation is deformed to a generic quadratic, the Calabi-Yau hypersurface will deform to a generic $(2,4)$ complete intersection in $\mathbb{P}^5$.  Morrison then gives a birational morphism between the mirror of the $(2,4)$ family and a subfamily of the mirror to the $\mathbb{P}^{(1,1,2,2,2)}$ family, which represents the mirror transition.

According to our approach, the mirror transition should be induced by the linear dual of the map that torically embeds $\mathbb{P}^{(1,1,2,2,2)}$ into $\mathbb{P}^5$.  A fan for $\mathbb{P}^{(1,1,2,2,2)}$ is contained in $M_\mathbb{R} \cong \mathbb{R}^4$ and given by the cones over the faces of the reflexive polytope $P$ whose vertices are $(-1,-2,-2,-2)$ and $f_1$, $f_2$, $f_3$, $f_4$ (the standard basis).  A toric embedding is given by the linear map $h : M_\mathbb{R} \rightarrow M'_\mathbb{R}$ which maps $\Sigma(P)$ to the fan $\Sigma(\Delta)$ for $\mathbb{P}^5$ (the same fan used previously), and has the matrix 
\[ \left( \begin{array}{cccc}
0 & 0 & 0 & 1 \\
0 & 0 & 1 & 0 \\
0 & 1 & 0 & 0 \\
1 & 0 & 0 & 0 \\ 
2 & 0 & 0 & 0 \\
\end{array} \right).\]
The dual map $h^* : N'_\mathbb{R} \rightarrow N_\mathbb{R}$ is given by the transpose.  
The lattice points contained in $P$ are its vertices along with the point $(0,-1,-1,-1)$, so by Batyrev's construction the mirror family is given by (a Calabi-Yau compactification of)
$$-1+a_1X_1+a_2 X_2 + a_3 X_3+a_4 X_4 + a_5 X^{-1}_1 (X_2 X_3 X_4)^{-2}+a_6 (X_2 X_3 X_4)^{-1} = 0$$
in $T= \hbox{Spec} \ \mathbb{C}[M]$, where $X_i = z^{f_i}$.  We identify the degenerate subfamily on the mirror as the subfamily satisfying $4a_1a_5 = a^2_6$, and refer to it as $\mathcal{X}^*_0$.  (This is the same as the condition $q_2 = 4$ in Morrison's description.)  As usual, by using the torus action we may assume without loss that $a_1 = \cdots = a_4 = 1$, and the degeneracy condition becomes $4a_5 = a^2_6$.  We can factor to obtain 
\begin{equation}
\label{factoredform}
-1+X_1(1+(a_6/2)(X_1X_2X_3X_4)^{-1})^2+X_2+X_3+X_4 = 0.
\end{equation}

For the mirror family to $(2,4)$ complete intersections in $\mathbb{P}^5$ we again take a Calabi-Yau compactification of the equations 
\begin{align*}
b_1Y_1+b_2Y_2+b_3Y_3+b_4Y_4 &= 1 \\
b_5Y_5+b_6(Y_1Y_2Y_3Y_4Y_5)^{-1} &= 1
\end{align*}
in $T'= \hbox{Spec} \ \mathbb{C}[M']$, where $Y_i = z^{e_i}$, with $e_1, \dots, e_5$ the standard basis of $M'$.  We can assume that all $b_i$ except $b_6$ equal 1 by using the torus action (note that in the previous example, the same equations were used, but we chose to set all coefficients other than $b_4$ to 1, rather than $b_6$).  Let us refer to this family with the single parameter $b_6$ as $\mathcal{X}^*_{(2,4)}$.  Then we have:

\begin{thm} After setting $b_6 = a_6/2$, the families $\mathcal{X}^*_{(2,4)}$ and $\mathcal{X}^{*}_0$ are birationally equivalent via the map the map $f: T' \rightarrow T$ given by $X_1 = Y_4 Y^2_5$, $X_2 = Y_3$, $X_3 = Y_2$, $X_4 = Y_1$.
\end{thm}

\begin{proof} Let $f_{b_6}$ be the restriction of $f$ to a single member of the family $\mathcal{X}^*_{(2,4)}$ with a particular value of $b_6 = a_6/2$.  By using the factored form of the equation for $\mathcal{X}^*_0$ (equation \eqref{factoredform}), it is straightforward to check that the rational map from $T$ to $T'$ defined by
\begin{align*}
Y_1 &= X_4 \\
Y_2 &= X_3 \\ 
Y_3 &= X_2 \\
Y_4 &= 1-X_2-X_3-X_4 \\
Y_5 &= (1+b_6(X_1X_2X_3X_4)^{-1})^{-1}
\end{align*}
is a birational inverse for $f_{b_6}$. \end{proof} 

This confirms our prediction, since the map $f$ is induced by the map $h^* : N'_\mathbb{R} \rightarrow N_\mathbb{R}$.  Like in the first example of quartic hypersurfaces in $P(2,4)$, one may ask whether the map $f$ can be extended to a larger domain, including more points of the ambient toric variety of the Batyrev-Borisov construction.  We will not attempt to answer this question fully, but make a few observations.  

First, the approach taken for the main $P(2,4)$ example in this paper cannot work in exactly the same fashion.  That is, there cannot exist MPCP subdivisions $\widehat{\Sigma}(P^*)$ and $\widehat{\Sigma}(\nabla)$ such that the map of fans $h^* : \widehat{\Sigma}(\nabla) \rightarrow \widehat{\Sigma}(P^*)$ exists.  This is because the kernel of $h^*$ is generated by $(0,0,0,2,-1)$, and while $(0,0,0,2,-1)$ is contained in $\nabla$, its negative $(0,0,0,-2,1)$ is not.  This means that if $F$ is any face of $\nabla$ intersecting the ray over $(0,0,0,-2,1)$, the cone over $F$ will be mapped by $h^*$ to a cone containing a nontrivial linear subspace, making a map of fans impossible.

In spite of this, it may be possible to throw out cones over badly-behaved faces $F$, and corresponding strata of the toric variety, and still obtain a map of compact Calabi-Yau families.  In general, many toric strata of the ambient toric variety will not intersect generic members of a Batyrev-Borisov family and can be removed.

Another difference between this example and the previous example is the behavior of the map $h^*$ on the reflexive polytopes $\nabla$ and $P^*$.  Recall by Lemma \ref{imagething} that we had $g^*(\nabla) = \Delta^*_{P(2,4)}$.  In this case, the corresponding equality $h^*(\nabla) = P^*$ will not be true.  

The ambient toric variety of $\mathcal{X}^*_{(2,4)}$ has a fan given by cones over the faces of the reflexive polytope $\nabla$, and $\nabla = Conv(Newt(\varphi_1), Newt(\varphi_2))$, where $\nabla$ and the piecewise linear functions $\varphi_i$ have the same definitions as before (in section 4).  Because $\varphi_1 \circ h^*$ is the piecewise linear function which equals one on the boundary of $P$, we have that $Newt(\varphi_1 \circ h^*) = P^*$.  However, $\varphi_2 \circ h^*$ is a function equaling 2 on the vertices $(-1,-2,-2,-2)$ and $(1,0,0,0)$ of $P$ and zero on all other vertices.  One can check that $Newt(\varphi_2 \circ h^*)$ is a polytope in $N_\mathbb{R}$ with vertices in the standard basis equal to the rows of 
\[ \left( \begin{array}{cccc}
-2 & 0 & 0 & 0 \\
2 & 0 & 0 & 0 \\
-2 & 2 & 0 & 0 \\
-2 & 0 & 2 & 0 \\ 
-2 & 0 & 0 & 2 \\
\end{array} \right).\]   
This polytope is not contained in $P^*$, so we do not have that $h^*(\nabla) = P^*$.  

Thus, although the dual map $h^*$ still induces a birational morphism between families, there are definite differences between the $P(2,4)$ case and the $\mathbb{P}^{(1,1,2,2,2)}$ case.  However, I expect that after deleting some parts of the ambient toric varieties on the mirror side, a similar approach could be used to construct a map between the resolved and compactified families $\mathcal{X}^*_{(2,4)}$ and $\mathcal{X}^*_C$.

{\it Acknowledgements.} I would like to thank the authors of the computer algebra program Macaulay2 \cite{M2} and its Polyhedra package \cite{birkner}.  I would also like to thank Nathan Ilten for helpful comments, and Mark Gross for introducing me to the problem of mirror symmetry for complete intersections in Grassmannians.

\bibliography{myrefs}

\begin{thebibliography}{10}

\bibitem{Ba}
V.~V. {Batyrev}.
\newblock {Dual Polyhedra and Mirror Symmetry for Calabi-Yau Hypersurfaces in
  Toric Varieties}.
\newblock {\em {J. Alg. Geom.}}, {\bf 3}:493--535, 1994.
\newblock arXiv:alg-geom/9310003.

\bibitem{BB}
V.~V. {Batyrev} and L.~A. {Borisov}.
\newblock {On Calabi-Yau Complete Intersections in Toric Varieties}.
\newblock December 1994.
\newblock arXiv:alg-geom/9412017.

\bibitem{BCKS}
V.~V. {Batyrev}, I.~{Ciocan-Fontanine}, B.~{Kim}, and D.~{van Straten}.
\newblock {Conifold Transitions and Mirror Symmetry for Calabi-Yau Complete
  Intersections in Grassmannians}.
\newblock October 1997.
\newblock arXiv:alg-geom/9710022.

\bibitem{birkner}
R.~Birkner.
\newblock {{\it Polyhedra}: A package for computations with convex polyhedral
  objects}.
\newblock {\em Journal of Software for Algebra and Geometry}, {\bf
  1}(1):11--15, 2009.

\bibitem{CLS}
D.~Cox, J.~Little, and H.~Schenk.
\newblock {\em Toric Varieties}.
\newblock American Mathematical Society, 2011.

\bibitem{M2}
Daniel~R. Grayson and Michael~E. Stillman.
\newblock Macaulay2, a software system for research in algebraic geometry.
\newblock Available at http://www.math.uiuc.edu/Macaulay2/.

\bibitem{GBB}
M.~{Gross}.
\newblock {Toric Degenerations and Batyrev-Borisov Duality}.
\newblock {\em {Math. Ann.}}, {\bf 333}:645--688, 2005.
\newblock arXiv:math/0406171.

\bibitem{mav}
A.~Mavlyutov.
\newblock {Degenerations and mirror contractions of Calabi-Yau complete
  intersections via Batyrev-Borisov mirror symmetry}.
\newblock 2011.
\newblock arXiv:0910.0793v2.

\bibitem{morrison}
David~R. Morrison.
\newblock {Through the Looking Glass}.
\newblock In {\em Mirror Symmetry III}, pages 263--277. American Mathematical
  Society and International Press, 1999.

\bibitem{nill}
B.~Nill.
\newblock {Gorenstein toric Fano varieties}.
\newblock {\em Manuscripta Math.}, {\bf 116}(2):183--210, 2005.

\bibitem{reid}
M.~Reid.
\newblock The moduli space of 3-folds with {$K = 0$} may nevertheless be
  irreducible.
\newblock {\em Math. Ann.}, {\bf 278}:329--334, 1987.

\bibitem{rossi}
Michele Rossi.
\newblock {Geometric Transitions}.
\newblock {\em J. Geom. Phys.}, {\bf 56}:1940--1983, 2006.
\newblock arXiv:math/0412514v1.

\end{thebibliography}
\bibliographystyle{plain}  
\end{document}